\documentclass[12pt,reqno]{amsart}

\usepackage{amsfonts,amsthm,amsmath}
\usepackage{amsaddr}
\newcommand{\numberset}{\mathbb} 
\newcommand{\R}{\numberset{R}}

\newcommand{\varphie}{\varphi_{\varepsilon, i}}
\newcommand{\thetae}{\theta_{\varepsilon}}

\usepackage[a4paper, total={125mm, 185mm}]{geometry}

\usepackage{verbatim}
\usepackage{setspace}

\makeatletter

\numberwithin{equation}{section}
\newtheorem{thm}{\indent\bf {Theorem}}[section]

\theoremstyle{definition}

\newtheorem{rmk}[thm]{\indent \textsc {Remark}}

\begin{document}

\def\author@andify{%
   \nxandlist {\unskip ,\penalty-1 \space\ignorespaces}%
     {\unskip {} }%
     {\unskip ,\penalty-2 \space }%
}
\title[Multipolar potentials and wHi]{Multipolar potentials and weighted Hardy inequalities}

\author[A. Canale]{Anna Canale}
\address{Dipartimento di Matematica, 
Università degli Studi di Salerno, Via Giovanni Paolo II, 132, 84084 Fisciano
(Sa), Italy.
}

\thanks{
The author is member of the Gruppo Nazionale per l'Analisi Matematica, 
la Probabilit\'a e le loro Applicazioni 
(GNAMPA) of the Istituto Nazionale di Alta Matematica (INdAM)}


\maketitle

\bigskip

\begin{abstract}

In this paper we state the following weighted Hardy type inequality for any functions $\varphi$ in a weighted Sobolev space 
and for weight functions $\mu$ of a quite general type
\begin{equation*}
c_{N,\mu}
\int_{\R^N}V\,\varphi^2\mu(x)dx\le
\int_{\R^N}|\nabla \varphi|^2\mu(x)dx
+C_\mu \int_{\R^N}W \varphi^2\mu(x)dx,
\end{equation*} 
where $V$ is a multipolar potential and $W$ is a bounded function from above depending on $\mu$. The method to get the result is based on the introduction of a suitable vector value function and on an integral identity that we state in the paper. We prove that the constant $c_{N,\mu}$ in the estimate is optimal by building a suitable sequence of functions.

\end{abstract}

\maketitle

\vfill\eject

\section{Introduction}
\bigskip

The paper fits into the context of multipolar weighted Hardy type inequalities in $\R^N$.

The main motivation to investigate weighted Hardy inequalities is the application of the estimates in the study of Kolmogorov operators

\begin{equation}\label{L Kolmogorov}
Lu=\Delta u+\frac{\nabla \mu}{\mu}\cdot\nabla u,
\end{equation}
defined on smooth functions, $\mu>0$ is a probability density on $\R^N$, perturbed by singular potentials and of the related evolution problems
$$
(P)\quad \left\{\begin{array}{ll}
\partial_tu(x,t)=Lu(x,t)+V(x)u(x,t),\quad \,x\in {\mathbb R}^N, t>0,\\
u(\cdot ,0)=u_0\geq 0\in L^2(\R^N, \mu(x)dx).
\end{array}
\right. $$
The classical Hardy inequality was introduced in 1920th \cite{Hardy20} in the one dimensional case (see also 
\cite{Hardy25, HLP}). For historical reviews see, for example, \cite{D, KMP, KO}.

The well known version in $L^2(\R^N)$, $N\ge 3$, of the inequality is 
\begin{equation}\label{iHi introd}
c_{o}(N)\int_{\R^N}\frac{\varphi^2}{|x|^2}dx
\le \int_{\R^N}|\nabla \varphi|^2 dx
\end{equation}
for any functions $\varphi\in H^1(\R^N)$, where $c_{o}(N)=\bigl(\frac{N-2}{2}\bigr)^2$ is the optimal constant. 
The inequality (\ref{iHi introd}) does not hold if $c_{o}(N)$ is replaced by any $c>c_{o}(N)$
(see, e.g., \cite{Garcia Peral, M, GKombe}).

In literature there exist reference papers 
in the case of Schr\"odinger operators with singular potentials 
$V(x)=\frac{c}{|x|^2}$, $c>0$. These potentials are interesting for the criticality: they lie at a borderline case where
standard theories such as the strong maximum principle and Gaussian bounds in \cite{Aronson} fail.  It does not belong to the Kato's class, then it cannot be regarded as a lower order perturbation term.

The operator $\Delta+\frac{c}{|x|^{2}}$ has the same homogeneity as the Laplacian.
In 1984 P. Baras and J. A. Goldstein in \cite{BarasGoldstein}
showed a remarkable result:  the evolution problem $(P)$ with $L=\Delta$ and $V(x)=\frac{c}{|x|^{2}}$ admits a unique positive solution
if $c\leq c_o=\left( \frac{N-2}{2} \right)^{2}$ and no positive solutions in the sense of distributions exist if $c>c_o$.
When it exists, the solution is
exponentially bounded, on the contrary, if $c>c_o$, there is the so-called instantaneous blow-up phenomenon.

An analogous result has been obtained by Cabr\'e and Martel in \cite{CabreMartel} for more general potentials $0\le V\in L_{loc}^1({\mathbb R}^N)$ and with a different approach. 
They showed that the boundedness  of
 $\lambda_1(\Delta+V)$, the bottom of the spectrum of the operator $-(L+\Delta)$,

\begin{equation*}
\lambda_1(L+V):=\inf_{\varphi \in H^1_\mu\setminus \{0\}}
\left(\frac{\int_{{\mathbb R}^N}|\nabla \varphi |^2\,dx
-\int_{{\mathbb R}^N}V\varphi^2\,dx}{\int_{{\mathbb R}^N}\varphi^2\,dx}
\right),
\end{equation*}
is a necessary 
and sufficient condition for the existence of positive exponentially bounded in time
solutions to the associated initial value problem. 

These results have been extended to Kolmogorov operators perturbed by inverse square potential $V=\frac{c}{|x|^2}$ in \cite{GGR, GGR2}.
Later, for more general drift term, similar result can be found in \cite{CGRT, CPT1}. 

If the potentials are of a more general type see 
\cite{CA Hardy type, CA improved, CA improved 2}. 

The case of the Schr\"odinger operator with multipolar inverse square potentials has been investigated in literature.

In particular, for the operator
$$
\mathcal{L}=-\Delta-\sum_{i=1}^n\frac{c_i}{|x-a_i|^2},
$$
$n\ge2$, $c_i\in \R$, for any $i\in \{1,\dots, n\}$, V. Felli, E. M. Marchini and S. Terracini in
\cite{FelliMarchiniTerracini} proved that the associated quadratic form
$$
Q(\varphi):=\int_{\R^N}|\nabla \varphi |^2\,dx
-\sum_{i=1}^n c_i\int_{{\mathbb R}^N}\frac{\varphi^2}{|x-a_i|^2}\,dx
$$
is positive if $\sum_{i=1}^nc_i^+<\frac{(N-2)^2}{4}$, $c_i^+=\max\{c_i,0\}$, 
conversely if
$\sum_{i=1}^nc_i^+>\frac{(N-2)^2}{4}$ there exists a configuration of poles such that $Q$ is not positive.
Later R. Bosi, J. Dolbeaut and M. J. Esteban in \cite{BDE} proved that 
for any $c\in\left(0,\frac{(N-2)^2}{4}\right]$ there exists 
a positive constant $K$ such that the multipolar Hardy inequality
\begin{equation*}
	c\int_{{\R}^N}\sum_{i=1}^n\frac{\varphi^2 }{|x-a_i|^2}\, dx\le 
	\int_{{\R}^N} |\nabla\varphi|^2 \, dx\\
	+ K \int_{\R^N}\varphi^2 \, dx 
\end{equation*}
holds for any $\varphi \in H^1(\R^N)$.
C. Cazacu and E. Zuazua in \cite{CazacuZuazua}, improving a result stated in 
\cite{BDE}, obtained the inequality with optimal constant

$$
\frac{(N-2)^2}{n^2}\sum_{\substack{i,j=1\\ i< j}}^{n}
	\int_{\R^N}\frac{|a_i-a_j|^2}{|x-a_i|^2|x-a_j|^2}\varphi^2\,dx
	\le\int_{\R^N}|\nabla \varphi|^2\,dx,
$$
for any $\varphi\in H^1(\R^N)$.
 
For Ornstein-Uhlenbeck type operators 
$$
Lu=\Delta u - \sum_{i=1}^{n}A(x-a_i)\cdot \nabla u,
$$
perturbed by multipolar
inverse square potentials 
\begin{equation}\label{V CPT2}
V(x)=\sum_{i=1}^n \frac{c}{|x-a_i|^2},\quad c>0, \quad a_1\dots,a_n\in \R^N,
\end{equation}
weighted multipolar Hardy inequalities with optimal constant
and related existence and nonexistence results were stated in \cite{CP}, 
with $A$ a positive definite real Hermitian $N\times N$ matrix, $a_i\in \R^N$, $i\in \{1,\dots , n\}$. 
In such a case, the invariant measure for these operators is the Gaussian measure
$\mu_A (x) dx =\kappa e^{-\frac{1}{2}\sum_{i=1}^{n}\left\langle A(x-a_i), x-a_i\right\rangle }dx$,
with a normalization constant $\kappa$.  In the paper some results in \cite{GGR} have been extended to the multipolar case.  We remark that to get weighted Hardy inequalities involving multipolar potentials $V$ as in (\ref{V CPT2}), we need to overcome even the difficulties due to the mutual interaction among the poles.

In \cite{CPT2} these results have been extended to Kolmogorov operators with a more general drift term using different techniques and following Cabr\'e-Martel's approach to get existence and nonexistence results. 

In this paper we extend the inequality stated in \cite{CazacuZuazua} to the weighted case  improving a result stated in \cite {CPT2}. The potential that we consider is

\begin{equation}\label{V}
V(x)=\frac{1}{2}\sum_{\substack{i,j=1\\ i\ne j}}^{n}
\frac{|a_i-a_j|^2}{|x-a_i|^2|x-a_j|^2}.
\end{equation}
In particular we state the following inequality

\begin{equation}\label{ineq intro}
c_{N, \mu}\int_{\R^N} V \varphi^2 \, d\mu \le
\int_{\R^N}|\nabla\varphi|^2 \, d\mu+\int_{\R^N} W\varphi^2 \, d\mu
\end{equation}
for any $\varphi\in H^1(\R^N, \mu(x)dx)$, where $W$ is a bounded function from above which depends on the weight.  


The proof is based on the introduction of a suitable vector value function and on an integral identity that we state. 

As applications to PDE, we get the existence of positive solutions to (P), with $V$  as in (\ref{V}), following Cabr\'e-Martel's approach as in \cite{CPT2}. 
This because the Hardy inequality is related to the estimate of the bottom of the spectrum 
$\lambda_1(L+V)$. In \cite{CPT2} one can found existence and nonexistence results in the multipolar case when $V\in L^1_{loc}(\R^N)$, so the existence results can be extended to potential $V$ in (\ref{V}). 

The constant $c_{N, \mu}$ in (\ref{ineq intro}) is optimal. To prove the optimality 
we reason as in \cite{CazacuZuazua}, but the sequence involved to prove the reverse inequality is different. 

The paper is organized as follows.

Section 1 is devoted to the introduction of the weight functions and to the weighted Hardy type inequality. In Section 2 we state the optimality of the constant in the estimate.


\bigskip\bigskip

\section{Weighted inequality}

\bigskip

Let $\mu\ge 0$ be a weight function on $\R^N$. We define the weighted Sobolev space 
$H^1_\mu=H^1(\R^N, \mu(x)dx)$
as the space of functions in $L^2_\mu:=L^2(\R^N, \mu(x)dx)$ whose weak derivatives belong to
$L_\mu^2$.

The conditions on $\mu$ which we need are the following.

\begin{itemize}
\item[$H_1)$] 
\begin{itemize}
\item[$i)$] $\quad \sqrt{\mu}\in H^1_{loc}(\R^N)$;
\item[$ii)$]  $\quad \mu^{-1}\in L_{loc}^1(\R^N)$;
\end{itemize}
\item [$H_2)$] there exists constants $C_\mu, K_\mu\in \R$, $k_\mu>2-N$, such that 
it holds
\begin{equation}\label{cond W}
W:=-\sum_{i=1}^n\frac{\beta}{|x-a_i|^2}\left[(x-a_i)\cdot
\frac{\nabla\mu}{\mu}-K_\mu\right]\le C_\mu.
\end{equation}
\end{itemize}
The class of weight functions $\mu$ satisfying conditions $H_1)$ and $H_2)$ was considered in \cite{CPT2} to get a weighted multipolar Hardy inequality involving inverse square potentials of multipolar type (see also \cite{CPT1} for a similar conditions in the case of a single pole).

 Under the hypotheses $i)$ and $ii)$ in $H_1)$ the 
space $C_c^{\infty}(\R^N)$ is dense in $H_{\mu}^1$ (see e.g. \cite{T}).  
So we can regard
$H_{\mu}^1$ as the completion of $C_c^{\infty}(\R^N)$ with respect to the Sobolev norm
$$
\|\cdot\|_{H^1_\mu}^2 := \|\cdot\|_{L^2_\mu}^2 + \|\nabla \cdot\|_{L^2_\mu}^2.
$$
We need the density result in the proof of Theorem \ref{Thm inequality} below to get the result for any function in $H^1_\mu$. 
Actually, with the additional assumption

\begin{itemize}
\item[$H_3)$] 
\begin{equation*}
\lim_{\delta\to 0}\frac{1}{\delta^2}\int_{B(a_i,\delta)}\mu(x)dx=0 
\qquad \forall\, i\in\{1, \dots , n\},
\end{equation*} 
\end{itemize}
then $C_c^{\infty}(\R^N\setminus\{a_1, \dots, a_n\})$ is dense in $H^{1}_{\mu}$
(see \cite{CPT2}). 

A class of weight functions satisfying hypotheses $H_2)$ is the following
\begin{equation}\label{def mu}
\mu(x)=\prod_{i=1}^n\frac{1}{|x-a_1|^\gamma}
e^{-\delta\sum_{j=1}^n|x-a_j|^m},\qquad \delta\ge 0,\quad m\le 2, 
\end{equation}
for suitable values of $\gamma<N-2$ (see \cite{CPT2} for details). 

Let
\begin{equation}\label{def V}
V(x)=\frac{1}{2}\sum_{\substack{i,j=1\\ i\ne j}}^{n}
\frac{|a_i-a_j|^2}{|x-a_i|^2|x-a_j|^2}.
\end{equation}
We remark that when $x$ tends to the pole $a_i$ we get $V\sim \frac{n-1}{|x-a_j|^2}$.

The next result states a weighted Hardy type inequality involving the potential $V$ defined in (\ref{def V}). In \cite{CazacuZuazua} a similar inequality was obtained in the case of Lebesgue measure in a different way and under different hypotheses. 

\begin{thm}\label{Thm inequality}
Under conditions $H_1)$ and $H_2)$ the following inequality holds

\begin{equation}\label {ineq 1}
c_{N, \mu}\int_{\R^N} V \varphi^2 \, d\mu \le
\int_{\R^N}|\nabla\varphi|^2 \, d\mu+\int_{\R^N} W\varphi^2 \, d\mu
\end{equation}
for any $ \varphi\in H^1_\mu$,  where $\beta=\frac{N+K_\mu-2}{n}$ and 
$c_{N, \mu}=\frac{(N+K_\mu-2)^2}{n^2}$.

As a consequence we get

\begin{equation}\label {ineq 2}
c_{N, \mu}\int_{\R^N} V \varphi^2 \, d\mu \le
\int_{\R^N}|\nabla\varphi|^2 \, d\mu+C_\mu\int_{\R^N} \varphi^2 \, d\mu.
\end{equation}
\end{thm}

\begin{proof}

It is enough to prove (\ref{ineq 1})
for any $\varphi \in C_{c}^{\infty}(\R^N)$. 
By density argument we extend the result to functions $\varphi \in H^1_\mu$.

We set
$$
f=\prod_{i=1}^n f_i:=\prod_{i=1}^n\frac{1}{|x-a_i|^\beta},
\qquad \beta>0,
$$
and introduce the function
\begin{equation}\label {def F}
F(x)=-\frac{\nabla f}{f}\,\mu=-\frac{\nabla f_i}{f_i}\,\mu:=
\sum_{i=1}^n \beta\, \frac{x-a_i}{|x-a_i|^2} \mu.
\end{equation}
It is useful note that

\begin{equation}\label {delta f su f}
\begin{split}
\frac{\Delta f}{f}&=\sum_{i=1}^n \frac{\Delta f_i}{f_i}+
\sum_{\substack{i,j=1\\ i\ne j}}^{n}\frac{\nabla f_i}{f_i}\frac{\nabla f_j}{f_j}
\\&=
\sum_{i=1}^n\frac{\beta^2-\beta(N-2)}{|x-a_i|^2}+
\beta^2\sum_{\substack{i,j=1\\ i\ne j}}^{n}\frac{(x-a_i)\cdot (x-a_j)}{|x-a_i|^2|x-a_j|^2}
\\&=
\sum_{i=1}^n\frac{n\beta^2-\beta(N-2)}{|x-a_i|^2}-
\frac{\beta^2}{2} \sum_{\substack{i,j=1\\ i\ne j}}^{n} 
\frac{|a_i-a_j|^2}{|x-a_i|^2|x-a_j|^2}
 \\&=
\sum_{i=1}^n\frac{n\beta^2-\beta(N-2)}{|x-a_i|^2}-\beta^2 V
\end{split}
\end{equation}
since
\begin{equation*}
\begin{split}
\sum_{\substack{i,j=1\\ i\ne j}}^{n}\frac{(x-a_i)\cdot (x-a_j)}{|x-a_i|^2|x-a_j|^2}  &=
\sum_{\substack{i,j=1\\ i\ne j}}^{n}\frac{|x|^2-x\cdot a_i-x\cdot a_j+a_i\cdot a_j}
{|x-a_i|^2|x-a_j|^2}
\\&=
\sum_{\substack{i,j=1\\ i\ne j}}^{n}\frac{\frac{|x-a_i|^2}{2}+\frac{|x-a_j|^2}{2}-\frac{|a_i-a_j|^2}{2}}
{|x-a_i|^2|x-a_j|^2}
\\&=
\sum_{\substack{i,j=1\\ i\ne j}}^{n}\frac{1}{2}\left( \frac{1}{|x-a_i|^2}+\frac{1}{|x-a_j|^2}
-\frac{|a_i-a_j|^2}{|x-a_i|^2|x-a_j|^2}\right) 
\\&=
(n-1)\sum_{i=1}^{n}\frac{1}{|x-a_i|^2}
-\frac{1}{2}\sum_{\substack{i,j=1\\ i\ne j}}^{n}\frac{|a_i-a_j|^2}{|x-a_i|^2|x-a_j|^2}.
\end{split}
\end{equation*}
We observe that
$$
V=-\frac{\Delta f}{f}\qquad\hbox{if}\quad \beta=\frac{N-2}{n}.
$$ 
We start from the following integral and use (\ref{def F}) and (\ref{delta f su f}) to obtain
\begin{equation}\label{div F left}
\begin{split}
\int_{\R^N}{\rm div}&F \,\varphi^2 dx =
\int_{\R^N}\left(-\frac{\Delta f}{f}\mu+\left|\frac{\nabla f}{f}\right|^2\mu
-\frac{\nabla f}{f}\cdot \nabla\mu\right)\varphi^2\,dx
\\&=
\int_{\R^N}\Biggl[\sum_{i=1}^n\frac{\beta(N-2)-n\beta^2}{|x-a_i|^2}+\beta^2 V
+\sum_{i=1}^n\frac{n\beta^2}{|x-a_i|^2}
\\&-
\beta^2 V+\beta\sum_{i=1}^n\frac{(x-a_i)}{|x-a_i|^2}\cdot\nabla\mu\Biggr]\varphi^2 dx
\\&=
\int_{\R^N}\sum_{i=1}^n\left[\frac{\beta(N-2)}{|x-a_i|^2} \mu +
\beta\frac{(x-a_i)}{|x-a_i|^2}\cdot\nabla\mu\right]\varphi^2 dx
\end{split}
\end{equation}
by observing that the functions
 $F_j$, $\frac{\partial F_j}{\partial x_j}$, where 
$F_j=\beta \sum_{i=1}^{n}\frac{(x-a_i)_j}{|x-a_i|^2}\mu$,
belong to $\in L_{loc}^1(\R^N)$. Let us see in more detail. By means of
 H\"older's inequality and the 
classical Hardy's inequality, taking in mind hypothesis $i)$ in $H_2)$, 
for any $K$ compact set in $\R^N$ we get 
\begin{equation*}
\begin{split}
\int_K |F_j| \,dx&\le  \beta \sum_{i=1}^{n}\int_K \frac{\mu(x)}{|x-a_i|} \,dx\le
\beta\sum_{i=1}^{n}\left(\int_K \frac{\mu(x)}{|x-a_i|^2}\,dx \right)^{\frac{1}{2}}
\left(\int_K \mu(x)\,dx\right)^{\frac{1}{2}}
\\&\le
\frac{2\beta n} {N-2}\left(\int_K \left|\nabla\sqrt \mu\right|^2\,dx \right)^{\frac{1}{2}}
\left(\int_K  \mu(x) \,dx\right)^{\frac{1}{2}}.
\end{split}
\end{equation*}
For the partial derivative of $F_j$
$$
\frac{\partial F_j}{\partial x_j}(x) =
\sum_{i=1}^{n}\frac{\beta}{|x-a_i|^2}\left\{\left[1-2\frac{(x-a_i)_j^2}{|x-a_i|^2}\right]\mu+(x-a_i)_j\frac{\partial \mu}{\partial x_j}\right\}
$$ 
we obtain the estimate

\begin{equation*}
\begin{split}
\int_K \left|\frac{\partial F_j}{\partial x_j}\right|\,dx &\le
3\beta\sum_{i=1}^{n}\int_K \frac{\mu(x)}{|x-a_i|^2}\,dx+
\beta\sum_{i=1}^{n}\int_K \frac{|\nabla \mu|}{|x-a_i|}\,dx 
\\&\le
\frac{3\beta n} {c_o(N)}\int_K \left|\nabla{\sqrt \mu}\right|^2\,dx+
2\beta\sum_{i=1}^{n}\int_K \frac{\sqrt \mu}{|x-a_i|}|\nabla\sqrt \mu|\,dx
\\& \le
\frac{3\beta n} {c_o(N)}\int_K \left|\nabla{\sqrt \mu}\right|^2\,dx+
2\beta\sum_{i=1}^{n}\left(\int_K \frac{\mu(x)}{|x-a_i|^2}\,dx\right)^{\frac{1}{2}}
\left(\int_K \left|\nabla\sqrt \mu\right|^2\,dx \right)^{\frac{1}{2}}
\end{split}
\end{equation*}
where $c_o(N)= \frac{(N-2)^2}{4}$ is the best constant in Hardy's inequality.
Then we apply again Hardy's inequality on the right-hand side above.

Integrating by parts on the left-hand side in (\ref{div F left}), we obtain

\begin{equation}\label{div F right}
\begin{split}
\int_{\R^N}&{\rm div}F \varphi^2 \,dx=
2\int_{\R^N}\varphi\nabla\varphi\cdot \frac{\nabla f}{f}\,\mu \, dx
\\&  =
-2\int_{\R^N}\left|\nabla\frac{\varphi}{f}\right|^2 f^2 \, d\mu+
\int_{\R^N}|\nabla\varphi|^2 \, d\mu+
\int_{\R^N}\left|\frac{\nabla f}{f}\right|^2 \varphi^2 \, d\mu,
\end{split}
\end{equation}
explained by the fact that
$$
\left|\nabla\frac{\varphi}{f}\right|^2=\frac{|\nabla \varphi|^2}{f^2}+
\left|\frac{\nabla f}{f}\right|^2 \frac{\varphi^2}{f^2}-
2\varphi\nabla\varphi\cdot \frac{\nabla f}{f} \frac{1}{f^2}.
$$
So, putting together (\ref{div F left}) and (\ref{div F right}), we deduce that
\begin{equation}\label{identity 1}
\begin{split}
\int_{\R^N}&|\nabla\varphi|^2 \, d\mu =
\int_{\R^N}\left|\nabla\frac{\varphi}{f}\right|^2 f^2 \, d\mu-
\int_{\R^N}\frac{\Delta f}{f}\varphi^2 \, d\mu
\\&-
\int_{\R^N}\frac{\nabla f}{f}\cdot\frac{\nabla \mu}{\mu}\, d\mu=
\int_{\R^N}\left|\nabla\frac{\varphi}{f}\right|^2 f^2 \, d\mu
\\&+
\int_{\R^N}\sum_{i=1}^{n}\frac{\beta(N+K_\mu-2)-n\beta^2}{|x-a_i|^2}\varphi^2\,d\mu+\beta^2\int_{\R^N}V\varphi^2 \, d\mu
\\&+
\int_{\R^N}\sum_{i=1}^n\frac{\beta}{|x-a_i|^2}\left[(x-a_i)\cdot
\frac{\nabla\mu}{\mu}-K_\mu\right]\varphi^2 \, d\mu.
\end{split}
\end{equation}
For $\beta=\frac{N+K_\mu-2}{n}$ in (\ref{identity 1}), setting
$c_{N, \mu}=\frac{(N+K_\mu-2)^2}{n^2}$, 
we get the integral identity

\begin{equation}\label{identity 2}
\int_{\R^N}|\nabla\varphi|^2 \, d\mu =
\int_{\R^N}\left|\nabla\frac{\varphi}{f}\right|^2 f^2 \, d\mu+
c_{N, \mu}\int_{\R^N} V\, \varphi^2 \, d\mu-
\int_{\R^N} W\, \varphi^2 \, d\mu.
\end{equation}
Identity (\ref{identity 2}) implies that
\begin{equation*}
c_{N, \mu}\int_{\R^N} V\, \varphi^2 \, d\mu\le
\int_{\R^N}|\nabla\varphi|^2 \, d\mu+\int_{\R^N}W\varphi^2 \, d\mu.
\end{equation*}
For the class of functions $\mu$ satisfying hypothesis (\ref{cond W}) in $H_2)$, 
it follows the inequality (\ref{ineq 2}).

\end{proof}

\begin{rmk} 
In the proof of Theorem \ref{Thm inequality} the initial assumption 
$\varphi \in C_{c}^{\infty}(\R^N)$ can be replaced by 
$\varphi \in C_c^{\infty}(\R^N\setminus\{a_1, \dots, a_n\})$ under the additional condition
$H_3)$ in the hypotheses of the Theorem \ref{Thm inequality}.
\end{rmk}

\begin{rmk}
If $\mu=1$ and, then, $K_\mu,C_\mu=0$,  we obtain the estimate in \cite{CazacuZuazua} with constant $c_N=\frac{(N-2)^2}{n^2}$ on the left-hand side in place of $c_{N,\mu}$.
\end{rmk}

\begin{rmk}
We remark that the inequality (\ref{identity 1}) allows us to have

\begin{equation}\label{ineq beta}
\begin{split}
\int_{\R^N}&\sum_{i=1}^{n}\frac{\beta(N+K_\mu-2)-n\beta^2}{|x-a_i|^2}\varphi^2\,d\mu
\\&+
\frac{\beta^2}{2}\int_{\R^N}\sum_{\substack{i,j=1\\ i\ne j}}^{n}\frac{|a_i-a_j|^2}{|x-a_i|^2|x-a_j|^2}\varphi^2 \, d\mu
\int_{\R^N}|\nabla\varphi|^2 \, d\mu+C_\mu\int_{\R^N}\varphi^2 \, d\mu.
\end{split}
\end{equation}
The maximum value of the first constant on left-hand side in (\ref{ineq beta}) is
$c_{N,n,\mu}=\frac{(N+K_\mu-2)^2}{4n}$ attained for $\beta=\frac{N+K_\mu-2}{2 n}$.
So we get the inequality 

\begin{equation}
\begin{split}
c_{N,n,\mu}\int_{\R^N}&\sum_{i=1}^{n}\frac{\varphi^2}{|x-a_i|^2}\,d\mu+
\frac{c_{N,n,\mu}}{2n}\int_{\R^N}\sum_{\substack{i,j=1\\ i\ne j}}^{n}\frac{|a_i-a_j|^2}{|x-a_i|^2|x-a_j|^2}\varphi^2 \, d\mu
\\&\le
 \int_{\R^N}|\nabla\varphi|^2 \, d\mu+C_\mu\int_{\R^N}\varphi^2 \, d\mu
\end{split}
\end{equation}
(cf. \cite{CPT2}). 

\end{rmk}

\bigskip\bigskip

\section{Optimality}

\bigskip

In this Section we prove the optimality of the constant on the left-hand side in (\ref{ineq 1}). 
(cf. \cite{CazacuZuazua} in the case of Lebesgue measure). 
To this aim we need further condition 
on the weight $\mu$. In particular we suppose that

\begin{itemize}
\item[$H_4)$] 
\begin{itemize}
\item[$i)$] $\quad$ 
for any $i\in \{1,\dots,n\}$, $n\ge 2$, it holds
$$
 \frac{\mu(x)}
{|x-a_i|^{\frac{2}{n}(N+k_\mu-2)+2}}\in L_{loc}^1(\R^N);
$$
\item[$ii)$] $\quad$
there exists a constant $C$ such that, for $|x|>\max_i |a_i|$, 
$$
\mu(x)\le \frac{C}{|x|^\gamma},\qquad \gamma> -(N+2 k_\mu-2), 
\quad 
$$
\end{itemize}
\end{itemize}
We observe that under hypothesis $i)$ in $H_4)$ we get 
$f=\prod_{i=1}^n|x-a_i|^{-\frac{N+K_\mu-2}{n}}\in H^1_{\mu, loc}$.

\begin{thm}\label{Thm optimality}
Let us assume hypothesis $H_4)$. Then the inequality (\ref{ineq 1}) does not hold for any $\varphi \in H^1_\mu$  if $c_{N, \mu}=\frac{(N+K_\mu-2)^2}{n^2}$ is replaced by any 
$c>c_{N, \mu}$.
\end{thm}

\begin{proof}

To state the optimality of the constant  $c_{N, \mu}$ and, then, the reverse inequality with respect to (\ref{ineq 1}) if $c>c_{N, \mu}$, it is sufficient there exists a sequence 
$\varphie\in H^1_\mu$ such that


\begin{equation}\label {lim fie}
\lim_{\varepsilon \rightarrow 0}
\Biggl[\int_{\R^N}|\nabla\varphie|^2 \, d\mu 
+\int_{\R^N}W\varphie^2 \, d\mu -
c_{N,\mu}\int_{\R^N} V \varphie^2\, d\mu\Biggr]=0.
\end{equation}
In order to prove (\ref{lim fie}), we use the integral inequality (\ref{identity 2}). So
we show that there exists a sequence 
$\varphie\in H^1_\mu$ such that
$$
\lim_{\varepsilon \rightarrow 0}
\int_{\R^N}\left|\nabla\frac{\varphie}{f}\right|^2 f^2\, d\mu=0.
$$
To this aim,
if $\displaystyle r_0\le\min_{\substack{1\le i,j\le n\\ i\ne j}} |a_i-a_j|/2$,
let $R$ such that $\bigcup_{i=1}^n B(a_i, r_0)\subset B(0,R)$ and
$\thetae\in C^\infty_c(\R^N)$, $\varepsilon\le 1$,  the following cut-off function 

$$
\thetae(x)= \left\{\begin{array}{ll}
1 \quad &\text{ if } |x|<\frac{R}{\varepsilon},\\
\in [0,1]\quad &\text{ if }  \frac{R}{\varepsilon}\le|x|\le\frac{2R}{\varepsilon},\\
0 \quad &\text{ if } |x|> \frac{2R}{\varepsilon},
\end{array}
\right. 
$$
Let us define the sequence $(\varphie)_{\varepsilon>0}$ as
$$
\varphie:= \thetae f
$$
which belongs to $\in H^1_\mu$.
We can choose $\varepsilon$ small enough so that 
\begin{equation}\label{choice e}
\varepsilon\le\min\left\{1,\frac{R}{c\max_i|a_i|}\right\},  
\qquad i\in\{1,\dots,n\}, \quad c>1.
\end{equation}
In $B_{\frac{2R}{\varepsilon}}(0)\setminus 
B_{\frac{R}{\varepsilon}}(0)$, away from the poles $a_i$, we get 
$|x|>c\max_i|a_i|$, so we have, using (\ref {choice e}),
$$
\frac{(c-1)}{c}\frac{R}{\varepsilon}\le |x|-|a_i|\le |x-a_i|\le |x|+|a_i|\le
\left(2+\frac{1}{c}\right)\frac{R}{\varepsilon}.
$$ 
Then, by condition $ii)$ in $H_4)$, we get
\begin{equation*}
\begin{split}
\int_{\R^N}\left|\nabla\frac{\varphie}{f}\right|^2 f^2\, d\mu&=
\int_{\R^N}\left|\nabla\thetae \right|^2
 \prod_{i=1}^n|x-a_i|^{-2\frac{N+K_\mu-2}{n}}\, d\mu
\\&=
\int_{B_{\frac{2R}{\varepsilon}}(0)\setminus 
B_{\frac{R}{\varepsilon}}(0)}\left|\nabla\thetae \right|^2
 \prod_{i=1}^n|x-a_i|^{-2\frac{N+K_\mu-2}{n}}\, d\mu
\\&\le
\frac{c_1}{R^2}\varepsilon^2
\left[\frac{c}{(c-1)R}\right]^{2(N+K_\mu-2)+\gamma}
{\varepsilon}^{2(N+K_\mu-2)+\gamma}
\int_{B_{\frac{2R}{\varepsilon}}(0)\setminus 
B_{\frac{R}{\varepsilon}(0)}}dx
\\&\le
c_2\,\omega_N\varepsilon^{2(N+K_\mu-1)+\gamma}
\int_{\frac{R}{\varepsilon}}^{\frac{2R}{\varepsilon}}\rho^{N-1}\,d\rho
\\&\le
c_3\,\omega_N{\varepsilon}^{(N+K_\mu-2)+K_\mu+\gamma}
\rightarrow 0 \quad\hbox{as}\quad \varepsilon \to 0,
\end{split}
\end{equation*}
where $\omega_N$ denotes the (N-1)-dimensional measure of the unit sphere in $\R^N$.

\end{proof}

\end{document}